\documentclass{amsart}
\swapnumbers 

\usepackage{latexsym}
\usepackage{amssymb}

\newtheorem{lemma}{Lemma}[section]
\newtheorem{corollary}[lemma]{Corollary}
\newtheorem{theorem}[lemma]{Theorem}

\newtheorem{proposition}[lemma]{Proposition}

\newtheorem{assumption}[lemma]{Standing Assumption}

\theoremstyle{definition} 
\newtheorem{definition}[lemma]{Definition}

\newtheorem{example}[lemma]{Example}
\newtheorem{examples}[lemma]{Examples}

\begin{document}

\title{Type-Decomposition of a Pseudo-Effect Algebra}

\author{David J. Foulis,
Sylvia Pulmannov\'{a}, Elena Vincekov\'a}

\date{}

\subjclass{Primary 06C15; Secondary 17C65, 46L45}

\keywords{(pseudo-)effect algebra, (pseudo-)MV-algebra, orthomodular
poset, orthomodular lattice, boolean algebra,
von Neumann algebra, JW-algebra, Loomis dimension lattice, types I,
II, and III.}

\thanks{The second and third authors were supported by Research and
Development Support Agency under the contract No. APVV-0071-06
and  LPP-0199-07; grant VEGA 2/0032/09,
Center of Excellence SAS - Quantum Technologies; ERDF OP R\&D Project CE QUTE ITMS
26240120009}


\begin{abstract}
The theory of direct decomposition of a centrally orthocomplete effect
algebra into direct summands of various types utilizes the notion of
a type-determining (TD) set. A pseudo-effect algebra (PEA) is a
(possibly) noncommutative version of an effect algebra. In this
article we develop the basic theory of centrally orthocomplete PEAs,
generalize the notion of a TD set to PEAs, and show that TD sets
induce decompositions of centrally orthocomplete PEAs into direct
summands.
\end{abstract}

\address{Emeritus Professor, Department of Mathematics and
Statistics, University of Massachusetts, Amherst, MA; Postal
Address: 1 Sutton Court, Amherst, MA 01002, USA;
\vskip 0.2cm
Mathematical
Institute, Slovak Academy of Sciences,  Stef\'anikova 49, SK-814 73
Bratislava, Slovakia}
\email{foulis@math.umass.edu,
pulmann@mat.savba.sk, vincek@mat.savba.sk}
\maketitle 

\markboth{\sc Foulis,
Pulmannov\'{a}, Vincekov\'a}{\sc Type Decomposition}

\section{Introduction} 

The classic theorem stating that a von Neumann algebra decomposes
uniquely as a direct sum of subalgebras of types I, II, and III,
\cite{MandvN}, \cite[I, \S 8]{Dix}, has played a prominent role both
in the development of the theory of von Neumann algebras and in
the applications of this theory in mathematical physics. Analogous
type-decomposition theorems were featured in subsequent work on
various generalizations of von Neumann algebras, including studies
of AW*-algebras \cite{Kap51}, Baer *-rings \cite{Kap68}, and JW-algebras
\cite{Top65}. For a von Neumann algebra $A$, and for the aforementioned
generalizations thereof, the subset $P$ of all projections (self-adjoint
idempotents) in $A$ forms an orthomodular lattice (OML) \cite{Beran,
Kalm}, and the decomposition of $A$ into types induces a corresponding
direct decomposition of the OML $P$. Conversely, a direct decomposition
of $P$ yields a direct-sum decomposition of the enveloping algebra
$A$. These connections between direct-sum decompositions of $A$ and
direct decompositions of $P$ have motivated a number of studies
of direct decompositions of more general OMLs.

The type-decomposition theorem for a von Neumann algebra is dependent
on the von Neumann-Murray dimension theory; likewise, the early
type-decomposition theorems for OMLs were corollaries of the
lattice-based dimension theories of L. Loomis \cite{Loomis} and
of S. Maeda \cite{Maeda}. The work of Loomis and Maeda was further
developed by A. Ramsay \cite{Ram} who proved that an arbitrary
complete OML is uniquely decomposed into four special direct summands,
one of which can be organized into a Loomis dimension lattice. More
recent and considerably more general results on type-decomposition
based on dimension theory can be found in the monograph of K.
Goodearl and F. Wehrung \cite{GWdim}.

In \cite[\S 7]{Kalm86} G. Kalmbach, without employing lattice
dimension theory \emph{per se}, obtained decompositions of
an arbitrary complete OML into direct summands with various
special properties---moreover, Ramsay's fourfold decomposition
is a special case of Kalmbach's theory. In \cite{CCM}, J. Carrega,
G. Chevalier, and R. Mayet obtained the direct decompositions of
Kalmbach and Ramsay by methods more in the spirit of universal algebra.

In \cite{FoPuType}, the decomposition theory of Kalmbach, Carrega,
\emph{et al.} was extended to the class of centrally orthocomplete
effect algebras (COEAs) by employing the notion of a type-determining
(TD) set. Effect algebras \cite{FBEA} are very general partially
ordered algebraic structures, originally formulated as an algebraic
base for the theory of measurement in quantum mechanics. Special
cases of lattice-ordered effect algebras are OMLs and the MV-algebras
of C. Chang \cite{CCC}.

The notion of a (possibly) non-commutative effect algebra, called
a \emph{pseudo-effect algebra}, was introduced and studied in a
series of papers by A. Dvure\v{c}enskij and T. Vetterlein
\cite{DV1, DV2, Dv03}. Whereas a prototypic example of an effect
algebra is the order interval from $0$ to a positive element in
a partially ordered abelian group, the analogous interval in a
partially ordered non-commutative group forms a pseudo-effect
algebra.

We review the definition and some of the notation for a
pseudo-effect algebra $E$ in Section \ref{sc:PEAs}, and we
study the center $\Gamma(E)$ of $E$ in Section \ref{sc:Central}.
In Section \ref{sc:COPEAs}, we focus attention on centrally
orthocomplete pseudo-effect algebras (COPEAs) and define
the \emph{central cover} of an element in a COPEA. \emph{For the
remainder of the article, we assume that $E$ is a COPEA.} The
notion of a \emph{type-determining} (TD) subset of $E$ is
introduced in Section \ref{sc:TDsets}, it is shown that
TD subsets induce decompositions of $E$ into direct summands
of various types. The article ends with Section \ref{sc:Examples}
where the important idea of a \emph{type-class} of pseudo-effect
algebras is introduced and a number of pertinent examples of
type-classes and corresponding TD subsets of $E$ are presented. Examples of the corresponding decompositions are
given.

\section{Pseudo-effect algebras} \label{sc:PEAs} 
A partial algebra $(E;\,+,\,0,\,1)$, where $+$ is a partial binary operation
and $0$ and $1$ are constants, is called a \emph{pseudo-effect algebra}
(PEA) iff, for all $a,b,c\in E$, the following conditions hold for all $a,b,c
\in E$:
\begin{enumerate}
\item $a+b$ and $(a+b)+c$ exist iff $b+c$ and $a+(b+c)$ exist, and in
 this case $(a+b)+c=a+(b+c)$.
\item There is exactly one $d\in E$ and exactly one $e\in E$ such that
 $a+d=e+a=1$.
\item If $a+b$ exists, there are elements $d,e\in E$  such that $a+b=
 d+a=b+e$.
\item If $1+a$ or $a+1$ exists, then $a=0$.
\end{enumerate}
Suppose that $E$ is a pseudo-effect algebra. If $a,b\in E$, define
$a\leq b$ iff there exists an element $c\in E$ such that $a+c=b$;
then $\leq$ is a partial ordering on $E$ such that $0\leq a\leq 1$ for
all $a\in E$. It is possible to show that $a\leq b$ iff $b=a+c=d+a$ for
some $c,d\in E$. We write $c=:a\diagup b$ and $d=:b\diagdown a$. Then
$(b\diagdown a)+a=a+(a\diagup b)=b$, and $a=(b\diagdown a)\diagup b=b
\diagdown(a\diagup b)$.
If $a\leq b\leq c$, then
\begin{eqnarray*} & (c\diagdown a)\diagdown (b\diagdown a)=c\diagdown b;
\ (a\diagup b)\diagup(a\diagup c)=b\diagup c;\\
& (c\diagdown b)\diagup(c\diagdown a)=b\diagdown a;\ (a\diagup c)\diagdown
(b\diagup c)=a\diagup b.
\end{eqnarray*}
We define $x^-:=1\diagdown x$ and $x^{\sim}:=x\diagup 1$ for any $x\in E$.
For a given element $e\in E$, we denote the order interval from $0$ to
$e$ by $E[0,e]:=\{x\in E: 0\leq x\leq e\}$ and we define the partial
binary operation $+_e$ on $E[0,e]$ as follows: for $f,g\in E[0,e]$,
$f+_e g$ exists iff $f+g$ exists in $E$ and $f+g\in E[0,e]$, in which
case $f+_e g=f+g$. Then $(E[0,e];\,+_e,\,0,\,e)$ is a pseudo-effect algebra.
For any $x\in E[0,e]$ we have $x^{-_e}:=e\diagdown x$, $x^{\sim_e}
:=x\diagup e$, and $e=x^{-_e}+x=x+x^{\sim_e}$.

If $a,b\in E$, we write an existing least upper bound (respectively,
greatest lower bound) of $a$ and $b$ in the partially ordered set $E$
as $a\vee b$ (respectively, as $a\wedge b$). Similarly, $\bigvee\sb
{i\in I}e\sb{i}$ and $\bigwedge\sb{i\in I}e\sb{i}$ denote, respectively,
the least upper bound in $E$ (if it exists) and the greatest lower bound
in $E$ (if it exists) of a family $(e\sb{i})\sb{i\in I}\subseteq E$.
Elements $a,b\in E$ are \emph{disjoint} iff $a\wedge b=0$. We say that
$E$ is \emph{lattice-ordered} iff $a\vee b$ and $a\wedge b$ exist
in $E$ for all $a,b\in E$.

\begin{example} \label{ex:IntervalPEA}
Let $G$ be any partially ordered (not necessarily abelian)
additively-written group, choose any element $0\leq u\in G$, and
let $G[0,u]=\{g\in G:0\leq g\leq u\}$. Then $(G[0,u];\,+,\,0,\,u)$
is a pseudo-effect algebra if we restrict the group operation $+$ to
$G[0,u]$.
\end{example}
If $x_1,x_2,\ldots,x_n$ are elements of a pseudo-effect algebra $E$,
we define the \emph{orthosum} $x_1+x_2+\cdots +x_n$  by recurrence:
$x_1+x_2+\cdots +x_n$ exists iff $x_1+x_2+\cdots +x_{n-1}$ and
$(x_1+x_2+\cdots+x_{n-1})+x_n$ exists, in which case we put $x_1+
x_2+\cdots +x_n :=(x_1+x_2+\cdots +x_{n-1})+x_n$. Owing to
associativity, we may omit parentheses, but the order of elements
is important.

Let $E$ and $F$ be pseudo-effect algebras. A mapping $\phi:E\to F$ is
a \emph{morphism} of pseudo-effect algebras (PEA-morphism) iff $\phi(1_E)
=1_F$ (where $1_E$ and $1_F$ are the unit elements in $E$ and $F$,
respectively), and $\phi(a)+\phi(b)$ exists whenever $a+b$ exists, with
$\phi(a+b)=\phi(a)+\phi(b)$. A morphism is an \emph{isomorphism} of
pseudo-effect algebras (PEA-isomorphism) iff it is a bijection and
$\phi^{-1}$ is also a morphism.

For more about basic properties of pseudo-effect algebras see \cite{DV1,DV2}.

\section{Central elements of pseudo-effect algebras} \label{sc:Central}


\begin{assumption}
In the sequel, $(E;\,+,\,0,\,1)$ is a pseudo-effect algebra.
\end{assumption}

\begin{definition}\label{de:central} {\rm \cite[Definition 2.1]{Dv03}} An
element $c$ of $E$ is said to be \emph{central} if there exists an
isomorphism\footnote{With coordinatewise operations, the cartesian product
of PEAs is again a PEA.}
$$
f_c:E\to E[0,c]\times E[0,c^{\sim}]
$$
such that $f_c(c)=(c,0)$ and if $f_c(x)=(x_1,x_2)$, then $x=x_1+x_2$ for
any $x\in E$.
\end{definition}
We denote by $\Gamma(E)$ the set of all central elements of $E$, and
we refer to $\Gamma(E)$ as the \emph{center} of $E$. Clearly, $0,1\in
\Gamma(E)$. In the next proposition, we collect some properties of
central elements (see \cite[Propositions 2.2, 2.4, and 2.5]{Dv03}).
\begin{proposition}\label{pr:central} Let $c$ be a central element of
$E$, and let $f_c$ be the corresponding mapping from Definition
\ref{de:central}. Then, for all $x,y,x_1,x_2\in E$:
\begin{enumerate}
\item $f_c(c^{\sim})=(0,c^{\sim})$.
\item If $x\leq c$, then $f_c(x)=(x,0)$.
\item $c\wedge c^{\sim}=0$.
\item If $y\leq c^{\sim}$ then $f_c(y)=(0,y)$.
\item $c^{\sim}=c^-$.
\item $x\wedge c\in E$, $x\wedge c^{\sim}\in E$, and
$$
 f_c(x)=(x\wedge c, x\wedge c^{\sim}).
$$
\item If $f_c(x)=(x_1,x_2)$, then $x=x_1\vee x_2$, $x_1\wedge x_2=0$,
 and $x_1+x_2=x$.
\item $x\wedge c=0$ iff $x\leq c^-$ iff $x\leq c^{\sim}$ iff
 $c\leq x^-$ iff $c\leq x^{\sim}$.
\item $c+c\in E$ implies $c=0$.
\item Let $c_1,c_2,\ldots,c_n\in \Gamma(E)$, $c_i\wedge c_j=0$ for
 $i\neq j$, and $c_1+c_2+\cdots +c_n=1$. Then $x=x\wedge c_1+x\wedge c_2
 +\cdots +x\wedge c_n$.
\end{enumerate}
\end{proposition}

In view of Proposition \ref{pr:central} (v), if $c\in \Gamma(E)$, then
we shall write $c\,':=c^-=c^{\sim}$. Also, we say that elements
$c,d\in \Gamma(E)$ are \emph{orthogonal} iff $c\wedge d=0$.

\begin{theorem}\label{th:central}{\rm \cite[Theorem 2.3]{Dv03}} If
$c,d\in \Gamma(E)$, then $c\wedge d$ exists in $E$ and belongs to
$\Gamma(E)$, and $\Gamma(E)=(\Gamma(E);\,\wedge,\,\vee,\,',\,0,\,1)$
is a Boolean algebra.
\end{theorem}

If $c\in \Gamma(E)$, then the mapping $p_c:E\to E[0,c]$ defined by
$$
p_c(x):=x\wedge c\text{\ for all\ }x\in E
$$
is a morphism from $E$ onto $E[0,c]$ whose kernel is $E[0,c']$.

\begin{proposition}\label{pr:p_e}{\rm \cite[Proposition 2.6]{Dv03}}
Let $x\in E$ and $c,d\in \Gamma(E)$. Then:
\begin{enumerate}
\item $p_{c\wedge d}=p_cp_d=p_dp_c$.
\item If $c\wedge d=0$, then $c+d=c\vee d=d+c$ and $p_{c\vee d}(x)=
 p_c(x)+p_d(x)=p_d(x)+p_c(x)$.
\item If $d\leq c$, then $c\diagdown d=c\wedge d\,'=d\diagup c$ and
 $p_{c\wedge d\,'}(x)=p_c(x)\diagdown p_d(x)=p_d(x)\diagup p_c(x)$.
\end{enumerate}
\end{proposition}

\begin{theorem}\label{th:sum-sup}{\rm \cite[Proposition 2.7]{Dv03}}
Let $c_1,c_2,\cdots+c_n\in \Gamma(E)$ with $c_i\wedge c_j=0$ for $i\neq j$.
Then:
\begin{enumerate}
\item $c:=\bigvee_{i=1}^n c_i=c_1+c_2+\cdots +c_n\in \Gamma(E)$, and
 for all $x\in E$,
$$
x\wedge c=\bigvee_{i=1}^n(x\wedge c_i)=x\wedge c_1+\cdots +x\wedge c_n.
$$
\item If $x_i\leq c_i$ for $i=1,2,\ldots,n$, then $x_1+x_2+\cdots +x_n
 =x_1\vee x_2\vee \cdots \vee x_n=x_{i_1}+x_{i_2}+\cdots +x_{i_n}$,
 where $(i_1,i_2,\ldots,i_n)$ is any permutation of $(1,2,\ldots,n)$.
\item If $a_1,a_2,\ldots a_n\in \Gamma(E)$, then for all $x\in E$,
$$
x\wedge(\bigvee_{i=1}^n a_i)=\bigvee_{i=1}^n(x\wedge a_i).
$$
\end{enumerate}
\end{theorem}

\begin{theorem}\label{th:theorem3.7} Suppose that $c_1,c_2,\ldots,c_n$
are pairwise orthogonal elements of $\Gamma(E)$ with $c_1+c_2+
\cdots+c_n=1$, let $X:=E[0,c_1]\times E[0,c_2]\times\cdots\times
E[0,c_n]$, and define $\Phi:X\to E$ by $\Phi(e_1,e_2,\ldots,e_n)
:=e_1+e_2+\cdots +e_n=e_1\vee e_2\vee \cdots \vee e_n$ for all
$(e_1,e_2,\ldots,e_n)\in X$. Then : {\rm(i)} $\Phi:X\to E$ is a
PEA-isomorphism. {\rm(ii)} If $e\in E$, then $\Phi^{-1}(e)=
(e\wedge c_1, e\wedge c_2,\ldots, e\wedge c_n)$.
\end{theorem}

\begin{proof}  If $(e_1,e_2,\ldots, e_n)\in X$, then $e_1+e_2+\cdots
+e_n=e_1\vee e_2\vee \cdots \vee e_n$ by Theorem \ref{th:sum-sup}
(ii). Clearly, $\Phi(1)=\Phi((c_1,c_2,\ldots,c_n))=c_1+c_2+\cdots+c_n
=1$. Assume that $e:=(e_1,e_2,\ldots, e_n), f:=(f_1,f_2,\ldots,f_n)\in X$
are such that $(e_1,e_2,\ldots,e_n)+(f_1,f_2,\ldots,f_n)=(e_1+f_1,
e_2+f_2,\ldots e_n+f_n)$ exists in $X$. Then $\Phi((e_1,e_2,\ldots,e_n))
=e_1+e_2+\cdots +e_n=e_1\vee \cdots \vee e_n$, $\Phi((f_1,f_2,
\ldots,f_n))=f_1+f_2+\cdots +f_n=f_1\vee \cdots \vee f_n$. Since $e_i
+f_i$ exists for $i=1,2,\ldots,n$, we have $e_i\leq f_i^-$ for $i=1,2,
\ldots,n$, and for $i\neq j$, we have $e_i\leq c_i, f_j\leq c_j$,
whence $e_i\leq c_i\leq c_j^-\leq f_j^-$. Then $\Phi(e)=\Phi((e_1,\ldots,e_n))
=e_1\vee e_2\vee \cdots \vee e_n\leq f_1^-\wedge f_2^-\wedge\cdots
\wedge f_n^-=\Phi(f)^-$, so that $\Phi(e)+\Phi(f)$ exists. Moreover, by associativity and Theorem \ref{th:sum-sup} (ii),
$$
\Phi((e_1,\ldots, e_n))+\Phi((f_1,\ldots f_n))=(e_1+e_2+\cdots +e_n)
+(f_1+f_2+\cdots +f_n)=
$$

$$
(e_1+f_1)+(e_2+f_2)+\cdots+(e_n+f_n)=
\Phi((e_1,e_2,\ldots,e_n)+(f_1,f_2,\ldots,f_n)).
$$

This shows that $\Phi$ is additive. For each $e\in E$, define $\Psi:
E\to X$ by $\Psi(e):=(e\wedge c_1, e\wedge c_2,\ldots, e\wedge c_n)
=(p_{c_1}(e),\ldots, p_{c_n}(e))$. Clearly, $\Psi(1)=1$ in $X$, and
if $e+f$ exists, then $\Psi(e+f)=\Psi(e)+\Psi(f)$, since $p_{c_i}$
are morphisms for all $i$.  Then $\Phi \circ \Psi(e)=e\wedge c_1+
e\wedge c_2+\cdots +e\wedge c_n=e$ by Proposition \ref{pr:central}
(x). If $(e_i)_{i=1}^n\subseteq X$, then $\Psi \circ \Phi((e_i)_{i=1}^n)=
\Psi((e_1+e_2+\cdots +e_n))=(p_{c_i}(e_1+\cdots +e_n))_{i=1}^n=
(e_i)_{i=1}^n$, since $p_{c_i}$, $i=1,2,\ldots,n$ is a morphism, and
$e_i\leq c_j$ for $i=j$, while $e_i\leq c_j\,'$ if $i\neq j$. It
follows that $\Phi^{-1}=\Psi$, and $\Psi$ is a morphism, hence $\Phi$
is an isomorphism.
\end{proof}

\begin{theorem}\label{th:subcentr}{\rm \cite[Proposition 2.8]{Dv03}}
For all $c\in \Gamma(E)$, $\Gamma(E[0,c])=\Gamma(E)[0,c]$.
\end{theorem}

\begin{lemma}\label{le:lemma3.9} Suppose that $e\in E$, $(f_i)_{i\in I}
\subseteq E$, $e+f_i$ {\rm(}respectively, $f_i+e${\rm)} exists for all
$i\in I$, and $\bigvee _{i\in I}f_i$ exists in $E$. Then $\bigvee_{i
\in I}(e+f_i)$ {\rm(}respectively, $\bigvee_{i\in I}(f_i+e)${\rm)}
exists in $E$ and $e+\bigvee_{i\in I}f_i=\bigvee_{i\in I}(e+f_i)$
\rm{(} respectively, $\bigvee_{i\in I}f_i+e=\bigvee_{i\in I}
(f_i+e)${\rm)}.
\end{lemma}

\begin{proof} Let $f:=\bigvee_{i\in I}f_i$. Assume that $e+f_i$
exists for all $i\in I$. Then $f_i\leq e^{\sim}$ for all $i\in I$,
so that $f\leq e^{\sim}$. Also $e+f_i\leq e+f$ for all $i\in I$.
Suppose that $r\in E$ and $e+f_i\leq r$ for all $i\in I$. It suffices
to prove that $e+f\leq r$. We have $e\leq e+f_i\leq r=e+(e\diagup r)$,
whence $f_i\leq e\diagup r$ for all $i\in I$, and it follows that
$f\leq e\diagup r$, hence $e+f\leq r$.
Now assume that $f_i+e$ exists for all $i\in I$. Then $f_i\leq e^-$,
whence $f\leq e^-$. Then $f_i+e\leq f+e$, and let $r\in E$ be such
that $f_i+e\leq r$ for all $i\in I$. Then $f_i\leq r\diagdown e$ for
all $i\in I$, whence $f\leq r\diagdown e$, and this implies $f+e
\leq r$.
\end{proof}

\begin{lemma}\label{le:lemma3.10} Suppose that $\phi :E\to E$ satisfies
the conditions $\phi(e)+f$ exists $\Rightarrow$ $e+\phi(f)$ exists,
and $f+\phi(e)$ exists $\Rightarrow$ $\phi(f)+e$ exists for all $e,f
\in E$. Then {\rm (i)} $\phi$ is order preserving. {\rm(ii)} If
$(e_i)_{i\in I}\subseteq E$ and $e:=\bigvee e_i$ exists in $E$, then
$\bigvee \phi(e_i)$ exists in $E$ and $\phi(e)=\bigvee_{i\in I}\phi(e_i)$.
\end{lemma}

\begin{proof} (i) Suppose $e\leq f$. Then $f^{\sim}\leq e^{\sim}$,
and as $\phi(f)+\phi(f)^{\sim}$ exists, then $f+\phi(\phi(f)^{\sim})$
exists $\Rightarrow$ $\phi(\phi(f)^{\sim})\leq f^{\sim}\leq e^{\sim}$
$\Rightarrow$ $e+\phi(\phi(f)^{\sim})$ exists $\Rightarrow$ $\phi(e)+
\phi(f)^{\sim}$ exists $\Rightarrow$  $\phi(e)\leq \phi(f)$.
(ii) Assume the hypothesis of (ii). As $e_i\leq e$, it follows from (i)
that $\phi(e_i)\leq \phi(e)$ for all $i\in I$. Suppose that $f\in E$
and $\phi(e_i)\leq f$ for all $i\in I$. Then $\phi(e_i)+f^{\sim}$ exists
$\Rightarrow$ $e_i+\phi(f^{\sim})$ exists $\Rightarrow$ $e_i\leq(\phi
(f^{\sim}))^-$ $\Rightarrow$ $e\leq (\phi(f^{\sim}))^-$ $\Rightarrow$
$e+\phi(f^{\sim})$ exists $\Rightarrow$ $\phi(e)+f^{\sim}$ exists
$\Rightarrow$  $\phi(e)\leq f$, proving (ii).
\end{proof}

\begin{theorem}\label{th:theorem3.11} Let $c\in \Gamma(E)$ and let
$(e_i)_{i\in I}$ be a family of elements of $E$. Then:

{\rm(i)} If $\bigvee_{i\in I}e_i$ exists in $E$, then $c\wedge
\bigvee_{i\in I}e_i=\bigvee_{i\in I}(c\wedge e_i)$.

{\rm(ii)} For every $e\in E$, $c=c\wedge e +c\wedge e^{\sim}$.
\end{theorem}

\begin{proof} (i) Define $\phi:E\to E$ by $\phi(e):=c\wedge e$
for all $e\in E$. Suppose $e,f\in E$ and assume that $\phi(e)+f$
exists. Then $c\wedge e\leq f^-\leq (c\wedge f)^-$. Also,
$c^-\wedge e\leq c^-\vee f^-=(c\wedge f)^-$, and by Proposition
\ref{pr:central} (vi) and (vii), $e=(c\wedge e)\vee(c\,'\wedge e)
\leq (c\wedge f)^-$, end consequently, $e+\phi(f)$ exists. Now assume
that $f+\phi(e)$ exists, then $c\wedge e\leq f^{\sim}\leq
(c\wedge f)^{\sim}$, and $c^{\sim}\wedge e\leq c^{\sim}\vee f^{\sim}
=(c\wedge f)^{\sim}$, and consequently $e=(c\wedge e)\vee (c\,'\wedge e)
\leq (c\wedge f)^{\sim}$, and so $\phi(f)+e$ exists. Therefore (i) follows
from Lemma \ref{le:lemma3.10}.

(ii) Put $e_1=e, e_2=e^{\sim}$. Then $e_1+e_2=1$, and  $c=p_c(e_1+e_2)
=p_c(e_1)+p_c(e_2)=c\wedge e+c\wedge e^{\sim}$.
\end{proof}

In the next theorem, we give an intrinsic characterization of central
elements. (For a similar result see \cite{YY}).
\begin{theorem}\label{th:charcent} An element $c$ in a PEA $E$ is
central if and only if the following properties are satisfied:
\begin{enumerate}
\item For all $a\in E$, there are $a_1,a_2\in E$, $a_1\leq c,\,
 a_2\leq c^{\sim}$ and $a=a_1+a_2$.
\item If $a,b\leq c$ {\rm(}respectively, $a,b\leq c^{\sim}${\rm)}
 and $a+b$ is defined, then $a+b\leq c$ \rm{(}respectively, $a+b
 \leq c^{\sim}${\rm)};
\item If $x,y\in E$, $x\leq c,\,y\leq c^{\sim}$, then $x+y=y+x$.
\end{enumerate}
\end{theorem}

\begin{proof} Observe first that (i)--(iii) imply that $c^{\sim}
=c^-$ and $c\wedge c^{\sim}=0$. Indeed, by (iii), $1=c+c^{\sim}
=c^{\sim}+c$, whence $c^{\sim}=c^-$. If $x\leq c,\,x\leq c^{\sim}$,
then by (ii), $c+x\leq c$, whence $x=0$.
If $c$ is central then property (i) follows by the definition of
central elements.\\ (ii): Let $a,b\leq c$, and $a+b$ exist. Then
$f_c(a)=(a,0),\,f_c(b)=(b,0)$ and $f_c(a+b)=(a,0)+(b,0)=(a+b,0)$,
hence $a+b\leq c$. Part (iii) follows by Theorem \ref{th:sum-sup}
(ii).

To prove the converse, define $f_c:E\to E[0,c]\times E[0,c^{\sim}]$
by $f_c(a)=(a_1,a_2)$ when $a=a_1+a_2,\,a_1\leq c,\,a_2\leq c^{\sim}$
by (i). We shall prove that $f_c$ satisfies Definition \ref{de:central} in the following steps.

(1) Assume that for $a\in E$, $a=a_1+a_2=b_1+b_2$ with $a_1,\,b_1\leq c,
\,a_2,b_2\leq c^{\sim}$ be two decompositions of $a$ by (i), and let
$a^{\sim}=d_1+d_2$, $d_1\leq c, d_2\leq c^{\sim}$ be any decomposition
of $a^{\sim}$. Then $1= a+a^{\sim}=(a_1+a_2)+(d_1+d_2)=a_1+(a_2+d_1)+
d_2$ by associativity. Since $a_2\leq c^{\sim}$ and $d_1\leq c$, we
obtain by (iii) that $a_2+d_1=d_1+a_2$. Again by associativity, $1=
(a_1+d_1)+(a_2+d_2)=c+c^{\sim}$, where $a_1+d_1\leq c$, $a_2+d_2
\leq c^{\sim}$ by (ii). It follows that $a_1+d_1=c$, $a_2+d_2=c^{\sim}$,
so $a_1=c\diagdown d_1$, $a_2=c^{\sim}\diagdown d_2$. Repeating this
reasoning with $a_1,\,a_2$ replaced by $b_1,\,b_2$, we obtain $a_1=
b_1,\,a_2=b_2$. This proves that $f_c$ is well defined.

Clearly, $f_c(c)=(c,0)$ and if $x\in E$ with $f_c(x)=(x_1,x_2)$,
then $x=x_1+x_2$.
If $f_c(a)=f_c(b)$, then $(a_1,a_2)=(b_1,b_2)$, which implies $a_1=b_1,\,
a_2=b_2$, whence $a=b$. This shows that $f_c$ is injective.

(2) Let $a,b\in E$ be such that $a+b$ exists. Let $f_c(a)=(a_1,a_2)$,
$f_c(b)=(b_1,b_2)$, and $f_c(a+b)=(d_1,d_2)$. Then $a=a_1+a_2$, $b=b_1
+b_2$, $a+b=d_1+d_2$. If follows that $(a_1+a_2)+(b_1+b_2)=a_1+b_1+a_2
+b_2=d_1+d_2$, using (iii). By (1) then, $d_1=a_1+b_1, d_2=a_2+b_2$.
Therefore $f_c(a+b)=(d_1,d_2)=(a_1+b_1,a_2+b_2)=(a_1,a_2)+(b_1,b_2)=
f_c(a)+f_c(b)$. This proves that $f_c$ is additive.

(3) Assume that $f_c(a)+f_c(b)$ exists in $E[0,c]\times E[0,c^{\sim}]$.
Then $(a_1,a_2)+(b_1,b_2)=(a_1+b_1, a_2+b_2)$, whence $a_1+b_1,\,a_2+b_2$
exist in $E$, and hence $(a_1+b_1)+(a_2+b_2)=a_1+a_2+b_1+b_2=a+b$. It
follows that $a+b$ exists iff $f_c(a)+f_c(b)$ exists.
To prove surjectivity, let $(x,y)\in E[0,c]\times E[0,c^{\sim}]$. Put
$z=x+y$, then $f_c(z)=(x,y)$.

Paerts (1),(2) and (3) imply that $f_c$ is a bijection such that $f_c(a)+
f_c(b)$ exists iff $a+b$ exists, and $f_c(a+b)=f_c(a)+f_c(b)$, hence
it is an isomorphism.
\end{proof}

\begin{lemma}\label{le:lemma3.13} If $p\in E$,  then: {\rm(i)} $c\in \Gamma(E)$
$\Rightarrow$ $p\wedge c\in \Gamma(E[0,p])$. {\rm(ii)} The mapping
$c\mapsto p\wedge c$ for $c\in \Gamma(E)$ is a boolean homomorphism
of $\Gamma(E)$ into the center $\Gamma(E[0,p])$ of $E[0,p]$.
\end{lemma}

\begin{proof} (i) Let $e\in E[0,p]$. Then $e=e\wedge c+e\wedge c\,'
=e\wedge p\wedge c+e\wedge p\wedge c\,'$.
From $p=p\wedge c+p\wedge c\,'=p\wedge c
\,'+p\wedge c$, we find that $p\wedge c\,'=(p\wedge c)\diagup p=p
\diagdown (p\wedge c)$. This implies that $(p\wedge c)^{\sim_p}=
(p\wedge c)^{-_p}=p\wedge c\,'$. Moreover, for every $e\in E[0,p]$
we have a decomposition $e=e_1+e_2$, $e_1\leq p\wedge c,\, e_2\leq p
\wedge c\,'=(p\wedge c)\diagup p$. If $e,f\leq p\wedge c$,
and $e+f$ exists in $E[0,p]$, then $e+f\leq p$, and $e,f\leq c$
implies $e+f\leq c$, hence $e+f\leq p\wedge c$. The same argument
holds if $e,f\leq (p\wedge c)\sp{\prime\sb{p}} :=p\wedge c\,'$. If
$x\leq p\wedge c,\, y\leq p\wedge c\,'$, then $x\leq c,\,y\leq c\,'$
implies that $x+y=y+x$. This proves that $p\wedge c\in\Gamma(E[0,p])$.
(ii) Part (ii) follows from Proposition \ref{pr:p_e}.
\end{proof}

\section{Centrally orthocomplete PEAs} \label{sc:COPEAs} 
\begin{definition}\label{de:gammaorthog} Two elements $p,q\in E$
are said to be \emph{$\Gamma$-orthogonal} iff there are orthogonal
central elements $c,\,d\in\Gamma(E)$ such that $p\leq c$ and $q
\leq d$. A family $(e_i)_{i\in I}$ is \emph{$\Gamma$-orthogonal}
iff there is a pairwise orthogonal family of elements
$(c_i)_{i\in I}\subseteq\Gamma(E)$ of central elements in $E$ such
that $e_i\leq c_i$ for all $i\in I$.
\end{definition}

Observe that, owing to Theorem \ref{th:sum-sup} (ii), if $e_1,\,e_2,
\ldots e_n$ are pairwise $\Gamma$-orthogonal elements, then their
orthosum exists and does not depend on the order of its summands;
moreover, $\sum_{i=}^n e_{i}=e_1+e_2+\cdots+e_n=e_1\vee e_2
\vee \cdots \vee e_n$.

\begin{definition} \label{df:COPEA}
Let $(e_i)_{i\in I}$ be a $\Gamma$-orthogonal family in $E$ and
let ${\mathcal F}$ be the collection of all finite subsets of the
indexing set $I$. Then $(e_i)_{i\in I}$ is \emph{orthosummable} iff
$$
\sum_{i\in I}e_i :=\bigvee\sb{F\in{\mathcal F}}\sum_{i\in F}e\sb{i}
$$
exists in $E$, in which case we refer to $\sum_{i\in I}e_i$ as the
\emph{orthosum} of the family. By definition, $E$ is a \emph{centrally
orthocomplete pseudo-effect algebra} (COPEA) iff every $\Gamma$-orthogonal
family in $E$ is orthosummable.
\end{definition}

\begin{lemma}\label{le:lemma4.3} {\rm (i)} If $e$ and $f$ are
$\Gamma$-orthogonal elements of $E$, then $e\leq f \ \Rightarrow \ e=0$.
{\rm(ii)} A family of central elements is $\Gamma$-orthogonal iff it is
pairwise orthogonal iff it is pairwise disjoint. {\rm(iii)} Every finite
$\Gamma$-orthogonal family in $E$ is orthosummable and its orthosum is
its supremum in $E$. {\rm(iv)} An arbitrary $\Gamma$-orthogonal family
in $E$ is orthosummable iff it has an orthosum iff it has a supremum in
$E$, and if it is orthosummable, then its orthosum coincides with its
supremum. {\rm(v)} $E$ is a COPEA iff every $\Gamma$-orthogonal family
in $E$ has a supremum in $E$.
\end{lemma}

\begin{proof} (i) If $e,f\in E$ and $c,d\in \Gamma(E)$ with $e\leq c$
and $f\leq d\leq c\,'$, then $e\leq f \ \Rightarrow \ e\leq c\wedge c\,'
=0$  by Proposition \ref{pr:central} (iii) and (v).
(ii) Follows directly from the definitions of $\Gamma$-orthogonality and
orthogonality of central elements.
(iii) Follows from Theorem \ref{th:sum-sup} (ii).
(iv) Follows by (iii) and the definition of the orthosum.
(v) Follows from (iv).
\end{proof}

\begin{assumption}
In the sequel, we assume that $E$ is a COPEA.
\end{assumption}

\begin{theorem}\label{th:theorem4.5} Let $(c_i)_{i\in I}$ be a pairwise
orthogonal family of elements in $\Gamma(E)$, and let $(e_i)_{i\in I}$,
$(f_i)_{i\in I}$ be families in $E$ such that $e_i,f_i\leq c_i$ and
$e_i+f_i$ exists for all $i\in I$. Then:
{\rm(i) } $c:=\sum_{i\in I}c_i=\bigvee_{i\in I} c_i$, $e:=\sum_{i
\in I}e_i=\bigvee_{i\in I}e_i\leq c$, $f:=\sum_{i\in I}f_i=\bigvee_{i
\in I}f_i\leq c$, and $e+f$ exists.
{\rm(ii)} $e+f=\sum_{i\in I}(e_i+f_i)=\bigvee_{i\in I}(e_i+f_i)\leq c$.
\end{theorem}

\begin{proof} (i) Part (i) follows from parts (ii) and (iv) of Lemma
\ref{le:lemma4.3}. For instance, the existence of $e+f$ is proved as
follows. As $e_i+f_i$ exists for all $i\in I$, we have $e_i\leq f_i^-$.
If $i\neq j$, then $e_i\leq c_i, f_j\leq c_j$, $c_i\wedge c_j=0$,
hence $e_i+f_j$ exists, so that $e_i\leq f_j^-$. Then $e=\bigvee_{i
\in I}e_i\leq f_j^-$ $\forall j\in I$, whence $e\leq \bigwedge_{j
\in I}f_j^-=(\bigvee_{j\in I}f_j)^-=f^-$, hence $e+f$ exists.

(ii) If $i\in I$, then $e_i,f_i\leq c_i$ implies that $e_i+f_i\leq c_i$
by Theorem \ref{th:charcent} (ii). From this it follows that
$(e_i+f_i)_{i\in I}$ is a $\Gamma$-orthogonal family in $E$, so by
Lemma \ref{le:lemma4.3} (iv) and (v),
$$
\sum_{i\in I}(e_i+f_i)=\bigvee_{i\in I}(e_i+f_i)\leq \bigvee_{i\in I} c_i=c.
$$
By Lemma \ref{le:lemma3.9}, $e+f=(\bigvee_{s\in I}e_s)+f=\bigvee_{s\in I}
(e_s+f)$, and for each $s\in I$, $e_s+f=e_s+\bigvee_{t\in I}f_t=
\bigvee_{t\in I}(e_s+f_t)$, and so
$$
\bigvee_{i\in I}(e_i+f_i)\leq \bigvee_{s,t\in I}(e_s+f_t)=e+f.
$$
Suppose $s,t\in I$. If $s=t$, then $e_s+f_t=e_s+f_s\leq \bigvee_{i
\in I}(e_i+f_i)$. If $s\neq t$, then $e_s+f_t\leq (e_s+f_s)+(e_t+f_t)
=(e_s+f_s)\vee(e_t+f_t)$. Consequently,
$$
e+f=\bigvee_{s,t\in I}(e_s+f_t)\leq \bigvee_{i\in I}(e_i+f_i).
$$
Combining the results obtained above, we get (ii).
\end{proof}

\begin{corollary}\label{co:cor4.6} Let $(c_i)_{i\in I}$ be a
pairwise orthogonal family of elements in $\Gamma(E)$ and let $d
\in E$. Put $c:=\bigvee_{i\in I}c_i$, $e:=\bigvee_{i\in I}(d
\wedge c_i)$, and $f:=\bigvee_{i\in I}(d^{\sim}\wedge c_i)$. Then:
{\rm(i)} $e\leq d$, $f\leq d^{\sim}$, and $c=e+f$.
{\rm(ii)} If $d\in E[0,c]$, then $d=\sum_{i\in I}(d\wedge c_i)
 =\bigvee_{i\in I}(d\wedge c_i)$.
\end{corollary}

\begin{proof} In Theorem \ref{th:theorem4.5}, let $e_i:=d\wedge c_i$
and $f_i:=d^{\sim}\wedge c_i$.
(i) As $e_i\leq d$ and $f_i\leq d^{\sim}$ for all $i\in I$, we get
$e=\bigvee_{i\in I}e_i\leq d$, and $f=\bigvee_{i\in I} f_i
\leq d^{\sim}$. By Theorem \ref{th:theorem3.11} (ii), $e_i+f_i=c_i$ for
all $i\in I$ , whence by Theorem \ref{th:theorem4.5} (ii), $e+f=
\bigvee_{i\in I}(e_i+f_i)=\bigvee_{i\in I} c_i=c$.
(ii) Assume that $d\in E[0,c]$. Then $e\leq d\leq c$ by (i). Thus
$e\leq (d^{\sim})^-$, hence $e+d^{\sim}$ exists, and $e+d^{\sim}=
\bigvee_{i\in I}(e_i+d^{\sim})$  by Lemma \ref{le:lemma3.9}. As $c_i\in
\Gamma(E)$, we have
$$
e_i+d^{\sim}=(d\wedge c_i)+d^{\sim}=(d\wedge c_i)+(d^{\sim}
  \wedge c_i)+(d^{\sim}\wedge c_i^{\sim})=c_i+(d^{\sim}
  \wedge c_i^{\sim})
$$

$$
=c_i\vee (d^{\sim}\wedge c_i^{\sim})=c_i\vee(d^{\sim}\wedge c_i)
\vee(d^{\sim}\wedge c_i^{\sim})= c_i\vee d^{\sim}, \text{\ so}
$$

$$
e+d^{\sim}=\bigvee_{i\in I}(d^{\sim}\vee c_i)\geq \bigvee_{i\in I}
(c^{\sim}\vee c_i)=c^{\sim}\vee c=1= e+e^{\sim}.
$$

By cancellation, $d^{\sim}\geq e^{\sim}$, whence $d\leq e$, and we
have $e=d$.
\end{proof}

\begin{theorem}\label{th:theorem4.7} {\rm(1)} Let $(c_i)_{i\in I}$ be a
pairwise orthogonal family of central elements, let $c:=\bigvee_{i
\in I} c_i$. Then $c\in \Gamma(E)$, and $\Gamma(E)$ is a complete
boolean algebra. {\rm(2)} For each $e\in E$ there is a smallest
element $d\in \Gamma(E)$ such that $e\leq d$.
\end{theorem}

\begin{proof}  (1) We have to prove properties (i)--(iii) of Theorem
\ref{th:charcent} for $c$.

(i) Let $d\in E$. By Corollary \ref{co:cor4.6}, $c=e+f$, $e\leq d$,
$f\leq d^{\sim}$. Then $d=e+e\diagup d$, and $e\diagup d=\bigvee_{i
\in I}(d\wedge c_i)\diagup d\leq d\wedge c_i\diagup d$ for all $i
\in I$. Let $x\in E$ be such that $x\leq d\wedge c_i\diagup d$ for
all $i\in I$. Then $d\wedge c_i+x\leq d$, hence $d\wedge c_i\leq d
\diagdown x$, so $\bigvee_{i\in I}(d\wedge c_i)\leq d\diagdown x$,
and therefore $x\leq \bigvee_{i\in I}(d\wedge c_i)\diagup d$. This
proves that $\bigvee_{i\in I}(d\wedge c_i)\diagup d=\bigwedge_{i\in I}
(d\wedge c_i\diagup d)=\bigwedge_{i\in I}d\wedge c_i^{\sim}\leq
\bigwedge_{i\in I}c_i^{\sim}=(\bigvee_{i\in I}c_i)^{\sim}=c^{\sim}$.
Finally we obtain $d=e+e\diagup d$, $e\leq c,\,e\diagup d\leq c^{\sim}$.
(ii) Let $e,f\leq c$ and suppose $e+f$ exists. Then $e_i:=e\wedge c_i
\leq c_i$, $f_i:=f\wedge c_i\leq c_i$, $(e_i)_{i\in I},\,(f_i)_{i\in I}$
are  $\Gamma$-orthogonal, and $e_i+f_i$ exists for all $i\in I$. By
Theorem \ref{th:theorem4.5}, $e=\bigvee_{i\in I} e_i$, $f=\bigvee_{i
\in I}f_i$, and $e+f=\bigvee_{i\in I}(e_i+f_i)\leq c$.
Let $e,f\leq c^{\sim}$ and suppose $e+f$ exists. From $c^{\sim}=
(\bigvee_{i\in I}c_i)^{\sim}=\bigwedge_{i\in I}c_i^{\sim}$ we obtain
that $e,f\leq c_i^{\sim}$ for all $i\in I$, and since $c_i$ is central,
$e+f\leq c_i^{\sim}$ for all $i\in I$. It follows that $e+f\leq
\bigwedge_{i\in I}c_i^{\sim}=c^{\sim}$.
(iii) Let $x,y\in E$, $x\leq c$, $y\leq c^{\sim}$. Then $x\wedge c_i
\leq c_i$, $y\leq c^{\sim}\leq c_i^{\sim}$ for all $i\in I$, and $x
=\bigvee_{i\in I} x\wedge c_i$ by Theorem \ref{th:theorem3.11}. Since $c_i$
is central, we have $x\wedge c_i+y=y+x\wedge c_i$, and by Lemma
\ref{le:lemma3.9}, $x+y=\bigvee_{i\in I}(x\wedge c_i+y)=\bigvee_{i\in I}
(y+x\wedge c_i)=y+x$. This proves (iii).
Therefore $c\in \Gamma(E)$, and by \cite[\S 20.1]{Sikor}, $\Gamma(E)$ is
a complete boolean algebra.

(2) Put $f=e^{\sim}$. Using Zorn's lemma we choose a maximal pairwise
orthogonal family $(c_i)_{i\in I}$ in $\Gamma(E)\cap E[0,f]$. As $c_i
\leq f$ for all $i\in I$, we have $c:=\bigvee_{i\in I}c_i\leq f$, and
$c\in \Gamma(E)$ by part (i) of this proof. Then $d:=c^-=\bigwedge_{i
\in I}c_i^-$, and $e=f^-\leq c^-=d\in \Gamma(E)$. To show that $d$ is
the smallest element in $\Gamma(E)$ such that $e\leq d$, let $e\leq k
\in \Gamma(E)$. Then $k^{\sim}\leq e^{\sim}=f$, so $k^{\sim}\wedge d\in
\Gamma(E)\cap E[0,f]$. Then $k^{\sim}\wedge d\leq d=c^-\leq c_i^-=
c\sb{i}\sp{\,\prime}$ for all $i\in I$, hence $k^{\sim}\wedge d$ is
orthogonal to all $c_i, i\in I$, and by maximality of $(c_i)_{i\in I}$,
$k^{\sim}\wedge d=k\,'\wedge d=0$. Since $k,d\in \Gamma(E)$, $d\leq k$,
which proves (2).
\end{proof}

\begin{definition}\label{de:centcov} If $e\in E$, then the smallest
element $d\in \Gamma(E)$ such that $e\leq d$ (Theorem \ref{th:theorem4.7}
(2) is called the \emph{central cover} of $e$, and we shall denote
it by $\gamma e:=d$.
\end{definition}

In the following definition, we extend the notion of a hull mapping
\cite{FoPuC, FoPuHD} to pseudo-effect algebras.
\begin{definition}\label{de:hull} A mapping $\eta\colon E\to\Gamma(E)$
such that {\rm(1)} $\eta 0=0$, {\rm(2)} $e\in E\ \Rightarrow \ e\leq
\eta e$, and {\rm(3)} $e,f\in E\ \Rightarrow \ \eta(e\wedge \eta f)
=\eta e\wedge \eta f$ is called a \emph{hull mapping} on $E$.
\end{definition}

\begin{theorem}\label{th:centcov} The central cover mapping $\gamma
\colon E\to\Gamma(E)$ is a surjective hull mapping\footnote{In
\cite{FoPuC}, a surjective hull mapping from an effect algebra
$E$ onto $\Gamma(E)$ (which is unique if it exists) is called a
\emph{discrete hull mapping}.} on $E$.
\end{theorem}

\begin{proof} Obviously, $\gamma 0=0$ and $e\leq \gamma e$ for all
$e\in E$. Let $e,f\in E$ and put $c:=\gamma f$. We have to prove that
$\gamma(e\wedge c)=\gamma e\wedge c$. Since $e\leq \gamma e$, we have
$e\wedge c\leq \gamma e\wedge c$, and hence $\gamma(e\wedge c)\leq
\gamma e\wedge c$. Since $c\in \Gamma(E)$, we have $e=(e\wedge c)
\vee(e\wedge c\,')\leq \gamma(e\wedge c)\vee c\,'\in \Gamma(E)$,
whence $\gamma e\leq \gamma(e\wedge c)\vee c\,'$. It follows that
$\gamma e\wedge c\leq \gamma(e\wedge c)\wedge c\leq \gamma(e\wedge c)$,
as desired. Since $\gamma(\gamma e)=\gamma(1\wedge \gamma e)=\gamma 1
\wedge\gamma e=\gamma e$, we obtain that $\gamma E:=\{\gamma e:e\in E\}
=\Gamma(E)$.
\end{proof}

\begin{lemma}\label{le:lemma4.11} Suppose that $(p_i)_{i\in I}\subseteq E$
is a $\Gamma$-orthogonal family in $E$. Let $p:=\bigvee_{i\in I}p_i$,
and let $c_i:=\gamma p_i$ for all $i\in I$ with $c=\bigvee_{i\in I} c_i$.
Then:
\begin{enumerate}
\item $p\leq\gamma p=c\in\Gamma(E)$.
\item $p\wedge c_i=p_i$ for all $i\in I$.
\item If $e\in E[0,p]$, then $e\wedge c_i=e\wedge p_i$ for all
 $i\in I$ and $e=\bigvee_{i\in I}(e\wedge p_i)$.
\end{enumerate}
\end{lemma}

\begin{proof} Since $(p_i)_{i\in I}$ is a $\Gamma$-orthogonal family,
$(c_i)_{i\in I}$ is an orthogonal family in $\Gamma(E)$, so $p$ and
$c$ are well-defined. Since $p_i\leq p$ for all $i\in I$, we have
$\bigvee_{i\in I}\gamma p_i=c\leq \gamma p$. On the other hand,
$p_i\leq \gamma p_i\leq c$ implies $\gamma p\leq c$. This proves (i).
Suppose that $i,j\in I$. If $i=j$, then $p_i\wedge c_i=p_i\wedge
\gamma p_i=p_i$; and if $i\neq j$, then $c_i\wedge c_j=0$, so $c_i
\wedge p_j=0$. Therefore, by Theorem \ref{th:theorem3.11} (i), $p
\wedge c_i=(\bigvee_{j\in I}p_j)\wedge c_i=\bigvee_{j\in I}(p_j
\wedge c_i)=p_i$, which proves (ii).
To prove (iii), suppose $e\in E[0,p]$. Then for each $i\in I$,
$e\wedge c_i=e\wedge p\wedge c_i=e\wedge p_i$ by (ii). Thus by
Corollary \ref{co:cor4.6} (ii), $e=e\wedge c=\bigvee_{i\in I}
(e\wedge c_i)=\bigvee_{i\in I}(e\wedge p_i)$.
\end{proof}

The following theorem extends Theorem \ref{th:theorem3.7} in the
setting of COPEAs. Since the proof is analogous to
\cite[Theorem 6.14]{FoPuC}, we omit it.

\begin{theorem}\label{th:theorem4.12} Let $(p_i)_{i\in I} \subseteq E$
be a $\Gamma$-orthogonal family in $E$, let $p:=\sum_{i\in I}p_i
=\bigvee_{i\in I}p_i$, and let $X:=\prod_{i\in }E[0,p_i]$. Define
the mapping $\Phi:X\to E[0,p]$ by
$$
\Phi((e_i)_{i\in I}):=\sum_{i\in I} e_i=\bigvee_{i\in I} e_i \
 \mbox{for every} \ (e_i)_{i\in I}\in X.
$$
Then $\Phi$ is a PEA-isomorphism of $X$ onto $E[0,p]$ and
$$
\Phi^{-1}(e):=(e\wedge \gamma p_i)_{i\in I}\ \mbox{for all} \ e\in E[0,p].
$$
\end{theorem}

\section{Type-determining sets} \label{sc:TDsets} 

\emph{The assumption that $E$ is a COPEA remains in force}.
As usual, a \emph{closure operator} on the set of all subsets
$Q$ of $E$ is a mapping $Q\mapsto Q^c$ such that, for all
$Q,R\subseteq E$, (1) $Q\subseteq Q^c$, (2) $Q\subseteq R \
\Rightarrow \ Q^c\subseteq R^c$, and (3) $(Q^c)^c=Q^c$. A subset
$Q$ is said to be \emph{closed} (with respect to $^c$) iff $Q^c=Q$.
The intersection of closed subsets is necessarily closed.
Generalizing the analogous notions for effect algebras in
\cite{FoPuType}, we introduce the following closure operators:
$Q\mapsto [Q]$, $Q\mapsto Q^{\gamma}$, $Q\mapsto Q^{\downarrow}$,
and $Q\mapsto Q''$, where
\begin{enumerate}
\item $[Q]$ is the set of all suprema of $\Gamma$-orthogonal
 families of elements of $Q$. We define $[\emptyset ]=\{0\}$.
\item $Q^{\gamma}:=\{ q\wedge c:q\in Q,\,c\in \Gamma(E)\}$.
\item $Q^{\downarrow}:= \bigcup_{q\in Q}E[0,q]$.
\item $Q':=\{ e\in E:q\wedge e=0 \ \forall q\in Q\}$.
\item $Q'':=(Q')'$.
\end{enumerate}

\begin{definition} We say that a subset $K\subseteq E$ is
\emph{type-determining} (TD) iff $K=[K]=K^{\gamma}$, and $K$
is \emph{strongly type-determining} (STD) iff $K=[K]=K^{\downarrow}$.
\end{definition}
Clearly, the intersection of TD (respectively, STD) subsets of
$E$ is again TD (respectively, STD).

\begin{theorem}\label{th:theorem5.2} Let $Q\subseteq E$. Then:
{\rm(i)} $[Q^{\gamma}]$ is the smallest TD subset of $E$ containing
$Q$. {\rm(ii)} $[Q^{\downarrow}]$ is the smallest STD subset of $E$
containing $Q$. {\rm(iii)} $Q'$ and $Q''$ are STD subsets of $E$.
{\rm(iv)} $Q'=[Q^{\gamma}]'=[Q^{\downarrow}]'$.
\end{theorem}

\begin{proof} Obviously, $Q\subseteq [Q^{\gamma}]$ and if $K$ is TD
and $Q\subseteq K$, then $[Q^{\gamma}]\subseteq K$. Also,
$[[Q^{\gamma}]]\subseteq [Q^{\gamma}]$, so to prove (i) it suffices
to show that $[Q^{\gamma}]^{\gamma}\subseteq [Q^{\gamma}]$. Let
$e\in[Q^{\gamma}]^{\gamma}$, then there exist $d\in \Gamma(E)$ and
$p\in [Q^{\gamma}]$ with $e=p\wedge d$. As $p\in [Q^{\gamma}]$, there
is a $\Gamma$-orthogonal family $(p_i)_{i\in I}\subseteq Q^{\gamma}$
with $p=\bigvee_{i\in I} p_i$, and for each $i\in I$, we can write
$p_i=q_i\wedge d_i$ with $q_i\in Q$ and $d_i\in \Gamma(E)$. Since
$e\leq p$, by Lemma \ref{le:lemma4.11} (iii), $e\wedge p_i$ exists for
all $i\in I$; moreover, $e\wedge p_i=p\wedge d\wedge p_i=p_i\wedge d
=q_i\wedge d_i\wedge d$. As $d_i\wedge d\in \Gamma(E)$, it follows
the $e\wedge p_i\in Q^{\gamma}$ for all $i\in I$, and the family
$(e\wedge p_i)_{i\in I}$ is $\gamma$-orthogonal.  Consequently, by
Lemma \ref{le:lemma4.11} (iii), $e=\bigvee_{i\in I}(e\wedge p_i)\in
[Q^{\gamma}]$. This proves (i). The proof of (ii) is quite similar to
the proof of (i), and we omit it.
To prove (iii), let $e\in Q'$ and $f\leq e$. Then $e\wedge q=0$ for
all $q\in Q$, whence $f\wedge q=0$ for all $q\in Q$, hence $f\in Q'$,
so that $Q'=Q'^{\downarrow}$.  Let $(p_i)_{i\in I}\subseteq Q'$ be
$\Gamma$-orthogonal family, and $p=\bigvee_{i\in }p_i$. Then $q
\wedge p_i=0$ for all $q\in Q$ and all $i\in I$, and since $q\wedge p
\leq p$, by Lemma \ref{le:lemma4.11} (iii), $p\wedge q=\bigvee_{i\in I}p
\wedge q\wedge p_i=0$, hence $p\in Q'$. It follows that $Q'=[Q']$,
and $Q'$ is STD. As $Q''=(Q')'$,  it follows that $Q''$ is STD.
To prove (iv), observe that $Q\subseteq [Q^{\gamma}]\subseteq
[Q^{\downarrow}]$ implies $[Q^{\downarrow}]'\subseteq [Q^{\gamma}]'
\subseteq Q'$. Let $e\in Q'$, and $(p_i)_{i\in I}$ be a
$\Gamma$-orthogonal family of elements in $Q^{\downarrow}$ with
$p=\bigvee_{i\in I}p_i$. Then each $p_i\leq q_i$ for some $q_i\in Q$,
and $e\wedge p_i\leq e\wedge q_i=0$ for all $i\in I$. By Lemma
\ref{le:lemma4.11}(iii), $e\wedge p=\bigvee_{i\in I}e\wedge p\wedge p_i
=0$, which shows that $e\in [Q^{\downarrow}]'$, proving (iv).
\end{proof}

\begin{theorem}\label{th:theorem5.3} Let $K\subseteq E$ be a TD set. Then:
{\rm(i)} $K\cap \gamma K=K\cap \Gamma(E)\subseteq \gamma K\subseteq
\Gamma(E)$. {\rm(ii)} There exists $c\in \Gamma(E)$ such that $\gamma K
= \Gamma(E)[0,c]$. {\rm(iii)} There exists $d\in \Gamma(E)$ such
that $K\cap \gamma K=\Gamma(E)[0,d]$.
\end{theorem}

\begin{proof} We omit the proof since it is analogous to the proof of
\cite[Theorem 4.5]{FoPuType}.
\end{proof}

Obviously, for every $c\in \Gamma(E)$, the central interval
$\Gamma(E)[0,c]=\Gamma(E)\cap E[0,c]$ is a TD subset of $E$.
\begin{corollary}\label{co:cor5.4} If $K$ is a TD subset of $E$,
then so are $\gamma K$  and $K\cap \gamma K$.
\end{corollary}

\begin{definition}\label{de:de4.7t} Let $K$ be a TD subset of $E$.
The (unique) element $c\in \gamma K$ such that $\gamma K=\Gamma(E)
[0,c]$ (Theorem \ref{th:theorem5.3} (ii)) is denoted by $c_K$ and is
called the \emph{type-cover} of $K$. The type cover $c_{K\cap
\gamma K}$ of the TD set $K\cap\gamma K$ is called the \emph{restricted
type-cover} of $K$.
\end{definition}

The following definition is analogous to \cite[Definition 5.1]{FoPuType}.
The terminology is borrowed from \cite[pp. 28--29]{Top65}.
\begin{definition}\label{de:de5.1t} Let $K$ be a TD subset of the
COPEA $E$ and let $c\in \Gamma(E)$. Then:
\begin{enumerate}
\item $c$ is type-$K$ iff $c\in K$.
\item $c$ is locally type-$K$ iff $c\in \gamma K$.
\item $c$ is purely non-$K$ iff no nonzero subelement of $c$ belongs to $K$.
\item $c$ is properly non-$K$ iff no nonzero central subelement of $c$
 belongs to $K$.
\end{enumerate}
\end{definition}

If $c\in \Gamma(E)$ and $c$ is type-$K$ (respectively, locally
type-$K$, etc.), we shall also say that the direct summand $E[0,c]$
of $E$ is type-$K$ (respectively, locally type-$K$, etc.).

The proof of the next theorem is omitted since it is the same
as the proof of \cite[Theorem 5.2]{FoPuType}.

\begin{theorem} \label{th:theorem5.7}
Let $K$ be a TD subset of $E$ and
let $c\in \Gamma(E)$. Then:
\begin{enumerate}
\item $c$ is type-$K$ $\Leftrightarrow$ $\Gamma(E)[0,c]\subseteq K\cap
  \gamma K$ $\Leftrightarrow$ $c\leq c_{K\cap \gamma K}$.
\item If $K$ is STD, then $c$ is type-$K$ $\Leftrightarrow$ $E[0,c]
 \subseteq K$.
\item $c$ is locally type-$K$ $\Leftrightarrow$ $\Gamma(E)[0,c]
 \subseteq\gamma K$ $\Leftrightarrow$ $c\leq c_K$.
\item $c$ is purely non-$K$ $\Leftrightarrow$ $K\cap E[0,c]=\{ 0\}$
 $\Leftrightarrow$ $c\leq (c_K)'$.
\item $c$ is properly non-$K$ $\Leftrightarrow$ $K\cap \Gamma(E)
 [0,c]=\{ 0\}$ $\Leftrightarrow$ $c\leq (c_{K\cap \gamma K})'$.
\item $c$ is both locally type-$K$ and properly non-$K$
 $\Leftrightarrow$ $c\leq c_K\wedge (c_{K\cap \gamma K})'$
\end{enumerate}
\end{theorem}

\begin{corollary}\label{co:co5.4t} If $K$ is a TD subset of $E$ and $c\in \Gamma(E)$, the following conditions are
equivalent: {\rm(i)} $c$ is locally type-$K$. {\rm(ii)} Every nonzero direct summand of $E[0,c]$ contains a nonzero element of $K$.
\end{corollary}
\begin{proof} (i) $\Rightarrow$ (ii). Assume (i). Then by Theorem \ref{th:theorem5.7} (iii), $\Gamma(E)[0,c]\subseteq \gamma K$, hence $c=\gamma k$ for some $k\in K$. Let $0\neq d\in\Gamma(E)[0,c]$. Then $k\wedge d\in K\cap E[0,d]$ with $\gamma(k\wedge d)=\gamma k\wedge d=c\wedge d=d\neq 0$, whence $k\wedge d\neq 0$.

(ii)$\Rightarrow$(i). Assume (ii). Then, $c\wedge c_K'\leq c$ an if $c\wedge c_K'\neq 0$, there exists $0\neq k\in K$ with $k\in E[0,c\wedge c_K']$, hence $\gamma k\leq c_K'$, contradicting Theorem \ref{th:theorem5.3} (ii). Therefore $c\wedge c_K'=0$, whence $c\leq c_K$.
\end{proof}

\begin{lemma}\label{le:lemma5.8} If $K$ is a TD subset of $E$, then
$c_{K'\cap \gamma(K')}=(c\sb{K})'$.
\end{lemma}

\begin{proof} We have to prove that $K'\cap\,\gamma(K')=\Gamma(E)
[0,(c\sb{K})']$. As $K'\cap\,\gamma(K')=K'\cap\,\Gamma(E)$,
it suffices to prove that, for $c\in \Gamma(E)$, $c\in K'\
\Leftrightarrow \ c\leq(c\sb{K})'$, the latter inequality
being equivalent to $c\wedge c_K=0$.
Let $c\in \Gamma(E)$. Suppose $c\in K'$ and let $k^*\in K$ be such
that $c_K=\gamma k^*$, then $c\wedge k^*=0$, whence $c\wedge c_K
=\gamma(c\wedge k^*)=0$. Conversely, suppose $c\wedge c_K=0$ and let
$k\in K$. Then, as $\gamma k\leq c_K$, it follows that $\gamma
(c\wedge k)=c\wedge \gamma k=0$, whence $c\wedge k=0$, so $c\in K'$.
\end{proof}

\begin{theorem}\label{th:theorem5.9} Let $K$ be a TD subset of $E$.
Then there exist unique pairwise orthogonal $c_1,c_2,c_3\in \Gamma(E)$
such that $c_1+c_2+c_3=1$;
$$
E=E[0,c_1]\times E[0,c_2]\times E[0,c_3];
$$
$c_1$ is type-$K$; $c_2$ is locally type-$K$, but properly non-$K$;
and $c_3$ is purely non-$K$. Moreover, $c_1=c_{K\cap \gamma K}$,
$c_2=c_K\wedge (c_{K\cap \gamma K})'$, $c_3=(c_K)'$,
$$
K\cap \gamma K=\Gamma(E)[0,c_1],\,K\subseteq E[0,c_1+c_2],\,
\Gamma(E)[0,c_2+c_3]\cap K=\{ 0\}.
$$
\end{theorem}

\begin{proof} Put $c_1:=c_{K\cap\gamma K}$, $c_2:=c_K\wedge (c_{K
\cap\gamma K})'$, and $c_3:=(c_K)'$. As $c_{K\cap\gamma K}\leq c_K$,
we have $c_1+c_2+c_3=1$, $c_1+c_2=c_K$, and $c_2+c_3=(c_{K\cap
\gamma K})'$. Thus, by part (i) of Theorem \ref{th:theorem5.7} (i),
$c_1$ is of type-$K$; by part (vi) of Theorem \ref{th:theorem5.7},
$c_2$ is locally type-$K$ and properly non-$K$, and by part (iv) of
Theorem \ref{th:theorem5.7}, $c_3$ is purely non-$K$.
To prove uniqueness, suppose that $c_1,c_2$ and $c_3$ satisfy the
conditions in the first part of the theorem. Then $c_1+c_2$ is locally
type-$K$, hence $c_1+c_2\leq c_K$, and $c_3$ is purely non-$K$, hence
$c_3\leq (c_K)'$ by Theorem \ref{th:theorem5.7} (iii) and (iv). Since
$c_1+c_2+c_3=1=c_K+(c_K)'$, we have $c_1+c_2=c_K$, and $c_3=(c_K)'$.
Moreover, $c_1$ is type-$K$, hence $c_1\leq c_{K\cap \gamma K}$, $c_2$
is locally type-$K$ but properly non-$K$, hence $c_2\leq c_K\wedge
(c\sb{K\cap\gamma K})'$. Since $c_1+c_2=c_K=c_{K\cap \gamma K}+c_K
\wedge(c\sb{K\cap\gamma K})'$, we obtain $c_1=c_{K\cap \gamma K}$,
$c_2=c_K\wedge(c_{K\cap\gamma K})'$.
\end{proof}

\section{Examples of TD sets and direct decompositions} \label{sc:Examples} 

Recall that an \emph{atom} in a pseudo-effect algebra $E$ is a nonzero
element $a\in E$ such that if $x\leq a$ then either $x=0$ or $x=a$.
A pseudo-effect algebra $E$ is \emph{atomic} iff for every $e\in E$
there is an atom $a\leq e$. Let $A$ (which may be empty) denote the
set of all atoms of $E$.

\begin{lemma}\label{le:atom} If $a\in A$ is an atom in $E$, then
$\gamma a$ is an atom in $\Gamma(E)$. Consequently, if $E$ is atomic,
then $\Gamma(E)$ is atomic.
\end{lemma}

\begin{proof} Let $a\in A$, and $c\in \Gamma(E)$, $c\leq \gamma a$.
Then $c=\gamma(c\wedge a)$, so that $c=0$ if $c\wedge a=0$, or $c=
\gamma a$ if $c\wedge a=a$. If $E$ is atomic, then for every $c\in
\Gamma(E)\subseteq E$ there is $a\in A$ with $a\leq c$, which yields
$\gamma a\leq c$.
\end{proof}

We say that an element $p\in E$, or equivalently, that $E[0,p]$ is
\emph{atom free} iff $A\cap E[0,p]=\emptyset$.

\begin{lemma} \label{le:lemma6.2} $[A]$ is the STD subset of $E$
generated by $A$.
\end{lemma}

\begin{proof} If $A=\emptyset$, then $A^{\downarrow}=\emptyset$,
otherwise $A^{\downarrow}=A\cup\{0\}$. In both cases,
$[A^{\downarrow}]=[A]$, and the result follows from Theorem
\ref{th:theorem5.2} (ii).
\end{proof}

An element of the STD set $[A]$ is called a \emph{polyatom}.
The following theorem for COPEAs is analogous to
\cite[Theorem 7.4]{FoPuType} for COEAs, and it enables us to
decompose $E$ into atomic and atom free parts.

\begin{theorem}\label{th:theorem6.3} {\rm(i)} The set $A'=[A]'$ is
STD and consists of all atom free elements of $E$. {\rm(ii)}
The set $A''=[A]''$ is STD and its nonzero part consists of
elements $p\in E$ such that $E[0,p]$ is atomic. {\rm(iii)}
$c_{A'\cap \gamma(A')}=c_{[A]}'$ is atom free. {\rm(iv)}
$A\subseteq [A]\subseteq E[0,c_{[A]}]$. {\rm(v)} If $p\in E$, then
$p$ is atom free iff $[A]\cap E[0,p]=\{ 0\}$. {\rm(vi)} $[A\cap
\Gamma(E)]=[A]\cap\Gamma(E)$.
\end{theorem}

\begin{proof} By Theorem \ref{th:theorem5.2} (iii), $A'$ and $A''$
are STD subsets of $E$. Since $p\in A'$ iff $p\wedge a=0$ for all
atoms $a\in A$, $A'$ is  the set of all atom free elements. Let $p
\in A''$, then $q\wedge a=0$ for all $a\in A$ implies $q\wedge p=0$,
hence if $p\wedge a=0$ for all $a\in A$, then $p=0$. Therefore if
$0\neq p\in A''$ then there is an atom $a\in A$ with $a\leq p$. This
proves (i) and (ii). Part (iii) follows from (i) and Lemma
\ref{le:lemma5.8}.
(iv) If $a$ is an atom, then $a=(a\wedge c_{[A]})+(a\wedge c_{[A]}')$,
where $a\wedge c_{[A]}'=0$ by part (iii). It follows that $a
\leq c_{[A]}$. Therefore, $A\subseteq E[0,c_{[A]}]$, and since
$E[0,c_{[A]}]$ is STD, $[A]\subseteq E[0,c_{[A]}]$.
(v) Every atom is a nonzero polyatom, and a polyatom is nonzero iff
it dominates an atom, hence $A\cap E[0,p]=\emptyset\ \Leftrightarrow
\ [A]\cap E[0,p]=\{ 0\}$.
(vi) Since $[A]$ is a TD subset of $E$, so is $[A]\cap \gamma [A]=
[A]\cap \Gamma(E)$. Thus, as $A\cap\Gamma(E)\subseteq [A]\cap\Gamma(E)$,
we have $[A\cap\Gamma(E)]\subseteq [A]\cap\Gamma(E)$. Let $h\in [A]
\cap \Gamma(E)$. Since $h\in [A]$, there is a $\Gamma$-orthogonal
sequence $(a_i)_{i\in I}$ of atoms with $h=\sum_{i\in I}a_i=
\bigvee_{i\in I} a_i$. Then $\gamma a_i$, $i\in I$, are pairwise
orthogonal elements in $\Gamma(E)$, and since $h\in \Gamma(E)$,
$h=\gamma h=\bigvee_{i\in I}\gamma a_i=\sum_{i\in I}\gamma a_i$. It
follows that $\sum_{i\in I} a_i=\sum_{i\in I}\gamma a_i$, and from
$a_i\leq \gamma a_i$ for all $i\in I$, we deduce that $a_i=\gamma a_i
\in \Gamma(E)$, and therefore $h\in [A\cap \Gamma(E)]$.
\end{proof}

The notions of boolean and subcentral elements and monads were
introduced in \cite{FoPuC}, and they also make sense in the
setting of pseudo-effect algebras.

\begin{definition}\label{de:bool} An element $b\in E$ is
\emph{boolean} iff $E[0,b]$ is a boolean algebra, i.e., $E[0,b]
=\Gamma(E[0,b])$.
\end{definition}

By Lemma \ref{le:lemma3.13}, for every $p\in E$ and $c\in \Gamma(E)$,
the element $p\wedge c$ is central in $E[0,p]$. The next definition
concerns those elements for which the converse also holds:
\begin{definition}\label{de:monsub} An element $p\in E$ is
\emph{subcentral} iff for every $d\in \Gamma(E[0,p])$, $d=
p\wedge c$ for some $c\in \Gamma(E)$.
\end{definition}
Clearly, every central element is subcentral (Theorem \ref{th:subcentr}), and every atom is
subcentral.

\begin{definition}\label{de:monad} An element $h\in E$ is a
\emph{monad} iff for every $e\in E[0,h]$, $e=h\wedge \gamma e$.
\end{definition}

Notice that every atom is a monad. Similarly as in \cite[Theorem 3.9]
{FoPuType}, we obtain the following characterization of monads.

\begin{theorem}\label{th:theorem6.7} Let $h\in E$. Then the
following are equivalent: {\rm(i)} $h$ is a monad. {\rm(ii)} $h$ is
both subcentral and boolean. {\rm(iii)} For all $e\in E[0,h]$,
$\gamma e=\gamma h \ \Rightarrow \ e=h$. {\rm(iv)} For all $e
\in E[0,h]$, $e^{\sim_h},\,e^{-_h}\leq (\gamma e)'$. {\rm(v)} For
all $e,f\in E[0,h]$, $e+_hf$ exists $\Leftrightarrow$
$\gamma e\wedge \gamma f=0$.
\end{theorem}

\begin{proof} (i) $\Rightarrow$ (ii). Let $h$ be a monad. Since
$\Gamma(E[0,h])\subseteq E[0,h]$, if $d\in \Gamma(E[0,h])$, then
$d=h\wedge\gamma d$, which shows that $h$ is subcentral. Since
$e\in E[0,h]$ implies $e=h\wedge \gamma e$, and $\gamma e\in \Gamma(E)$,
by Lemma \ref{le:lemma3.13}, $e$ is central in $E[0,h]$, hence $E[0,h]
=\Gamma(E[0,h])$, so $h$ is boolean.

(ii)$\Rightarrow$ (i). As $h$ is subcentral, every $d\in \Gamma(E[0,h])$
is of the form $d=h\wedge c$ for some $c\in \Gamma(E)$. Then $d\leq c$
implies $\gamma d\leq c$, and $d=d\wedge \gamma d=h\wedge c\wedge
\gamma d=h\wedge\gamma d$. Since $h$ is also boolean, $\Gamma(E[0,h])
=E[0,h]$, whence $d=h\wedge \gamma d$ holds for all $d\in E[0,h]$.

(i)$\Rightarrow$(iii). Assume $\gamma e=\gamma h$, $e\leq h$. Then
$e=h\wedge \gamma e=h\wedge \gamma h=h$.

(iii)$\Rightarrow$(iv). Assume (iii), let $e\in E[0,h]$ and put $f
:=e+(h\wedge(\gamma e)')$. As $e\leq \gamma e$, $h\wedge(\gamma e)'
\leq (\gamma e)'$, and $\gamma e\in \Gamma(E)$, it follows that
$f=e\vee (h\wedge(\gamma e)')\in E[0,h]$. Since $\gamma e\leq \gamma h$,
we have
$$
\gamma f=\gamma e\vee \gamma(h\wedge(\gamma e)')=\gamma e\vee
 (\gamma h\wedge (\gamma e)')=(\gamma h\wedge \gamma e)\vee
 (\gamma h\wedge (\gamma e)')=\gamma h,
$$
whence by (iii), $e+(h\wedge(\gamma e)')=f=h=e+e\diagup h$ and
it follows that $h\wedge(\gamma e)'=e\diagup h=e^{\sim_h}\leq
(\gamma e)'$. We can also write $f=(h\wedge(\gamma e)')+e=h=
h\diagdown e+e$, which yields $h\wedge(\gamma e)'=h\diagdown e
=e^{-_h}\leq(\gamma e)'$.

(iv)$\Rightarrow$(v). Let $e,f\in E[0,h]$, and assume that $e+_h f$
exists. Then $f\leq e^{\sim_h}\leq (\gamma e)'$, the last inequality
following from (iv). Now $f\leq (\gamma e)'$ implies $\gamma f\leq
(\gamma e)'$ which entails (v).

(v)$\Rightarrow$(i). Let $e\in E[0,h]$, then $h=e+e^{\sim_h}=e+
(e\diagup h)$, and by (v), $\gamma(e\diagup h)\leq (\gamma e)'$. We
also have $h=h\wedge \gamma e+h\wedge (\gamma e)'$, and from $e
\leq h\wedge \gamma e$ and $e\diagup h\leq h\wedge (\gamma e)'$ we
deduce that $e=h\wedge \gamma e$, whence $h$ is a monad.
\end{proof}

Let $S$ denote the set of all subcentral elements of $E$, $B$ the
set of all boolean elements of $E$ and $H$ the set of all monads
in $E$. As in \cite{FoPuC}, it can be shown that $S$ is a TD set with
$[A]\subseteq S$, $B$ is an STD set with $[A]\subseteq B$, and $H=S\cap B$
is an STD set with $[A]\subseteq H$.

The following definition is an analogue of \cite[Definition 4.2]
{FoPuType}.

\begin{definition}\label{de:typeclass} A nonempty class $\mathcal K$
of PEAs is called a \emph{type-class} iff the following conditions
are satisfied: {\rm(1)} $\mathcal K$ is closed under the passage to
direct summands, i.e., if $H\in {\mathcal K}$ and $h\in \Gamma (H)$,
then $H[0,h]\in {\mathcal K}$. {\rm(2)} $\mathcal K$ is closed under
the formation of arbitrary direct products. {\rm(3)} If $E_1$ and $E_2$
are isomorphic PEAs and $E_1\in {\mathcal K}$, then $E_2\in{\mathcal K}$.
If, in addition to {\rm(2)} and {\rm(3)}, ${\mathcal K}$ satisfies
{\rm(1')} $H\in {\mathcal K}$, $h\in H$ $\Rightarrow$ $H[0,h]\in
{\mathcal K}$, then ${\mathcal K}$ is called a \emph{strong type-class}.
\end{definition}

We omit the proof of the next theorem as it is analogous to the
proof of \cite[Theorem 4.4]{FoPuType}.

\begin{theorem}\label{th:theorem6.9} Let $\mathcal K$ be a type-class of COPEAs
and define $K:=\{k\in E:E[0,k]\in{\mathcal K}\}$. Then $K$ is a TD
subset of $E$. If $\mathcal K$ is a strong type-class, the $K$ is STD.
\end{theorem}

\begin{examples}
The class of effect algebras (EAs) and the following subclasses
of effect algebras are strong type-classes: all boolean EAs, all
OMLs, all complete OMLs, all orthoalgebras, all lattice EAs, and
all atomic EAs. Similarly, all lattice-ordered PEAs and all atomic
PEAs are strong type-classes.
\end{examples}

According to \cite{Dv03}, the PEA $E$ is (i) \emph{monotone
$\sigma$-complete} iff any ascending sequence $x_1\leq x_2\leq
\cdots$ in $E$ has a supremum $\bigvee_{i=1}^{\infty}x_i$ in
$E$; (ii) $E$ is \emph{$\sigma$-complete} iff it is a
$\sigma$-complete lattice; (iii) $E$ satisfies the \emph{countable
Riesz interpolation property} ($\sigma$-RIP) iff, for countable
sequences $\{ x_1,x_2,\ldots \}$ and $\{ y_1,y_2,\ldots \}$ of
elements of $E$ such that $x_i\leq y_j$ for all $i,j$, there
exists an element $z\in E$ such that $x_i\leq z\leq y_j$ for all
$i,j$; and (iv) $E$ is \emph{archimedean} iff the only $x\in E$
such that $nx:=x+\cdots+x$ is defined in $E$ for any integer
$n\geq 1$ is $x=0$.

One can easily deduce that the monotone $\sigma$-complete PEAs,
the $\sigma$-complete PEAs, the PEAs with the countable Riesz
interpolation property, and archimedean PEAS are all strong
type-classes.

In \cite{DV1}, the following properties of PEAs were introduced.
\begin{definition}\label{de:de3.1DV1}  Let $(E;\,+,\,0,\,1)$ be a
pseudo-effect algebra. Then:
\begin{enumerate}
\item $E$ fulfills the \emph{Riesz Interpolation
 Property} (RIP) iff, for any $a_1,a_2,b_1,b_2\in E$ such that
 $a_1,a_2\leq b_1,b_2$ there is $c\in E$ such that $a_1,a_2
 \leq c\leq b_1,b_2$.
\item $E$ fulfills the \emph{Weak Riesz Decomposition Property}
 (RDP$_0$) iff, for any $a,b_1,b_2\in E$ such that $a\leq b_1+b_2$,
 there are $d_1,d_2\in E$ such that $d_1\leq b_1,\,d_2\leq b_2$ and
 $a=d_1+d_2$.
\item $E$ fulfills the \emph{Riesz Decomposition Property} (RDP) iff,
 for any $a_1,a_2,b_1,b_2$ $\in E$ such that $a_1+a_2=b_1+b_2$ there are
 $d_1,d_2,d_3,d_4\in E$ such that $d_1+d_2=a_1$, $d_3+d_4=a_2$,
 $d_1+d_3=b_1$, and $d_2+d_4=b_2$.
\item $E$ fulfills the \emph{Commutational Riesz Decomposition Property}
 (RDP$_1$) iff, for any $a_1,a_2,b_1,b_2\in E$ such that $a_1+a_2=b_1+
 b_2$ there are $d_1,d_2,d_3,d_4\in E$ such that {\rm(1)} $d_1+d_2=a_1$,
 $d_3+d_4=a_2$, $d_1+d_3=b_1$, $d_2+d_4=b_2$ and {\rm(2)} $x\leq d_2,
 y\leq d_3$ imply $x+y=y+x$.
\item $E$ fulfills the \emph{Strong Riesz Decomposition Property}
 (RDP$_2$) iff, for any $a_1,a_2,b_1,b_2\in E$ such that $a_1+a_2=
 b_1+b_2$ there are $d_1,d_2,d_3,d_4\in E$ such that {\rm(1)} $d_1
 +d_2=a_1$, $d_3+d_4=a_2$, $d_1+d_3=b_1$, $d_2+d_4=b_2$ and {\rm(2)}
 $d_2\wedge d_3=0$.
\end{enumerate}
\end{definition}

\begin{proposition}\label{pr:pr3.3DV1} {\rm \cite[Proposition 3.3]{DV1}} Let $(E;+,0,1)$ be a pseudo-effect algebra.
{\rm(i)} We have the implications
$$
{\rm (RDP}_2{\rm )}\Rightarrow{\rm (RDP}_1{\rm )}\Rightarrow{\rm (RDP)}
 \Rightarrow{\rm (RDP}_0{\rm )}\Rightarrow{\rm (RIP)},
$$
The converse of any of these implications fails.

{\rm(ii)} $E$ fulfils (RDP$_2$) iff $E$ is lattice ordered and fulfils (RDP$_0$).

{\rm(iii)} Let  $E$ be commutative (i.e., an effect algebra)Then we have the implications

(RDP$_2$) $\Rightarrow$ (RDP$_1$)$\Leftrightarrow$ (RDP)
$\Leftrightarrow$ (RDP$_0$) $\Rightarrow$ (RIP).
Any implication not shown here does not hold.
\end{proposition}

Since for any $k\in E$, if $a+b$ exists in $E[0,k]$ then $a+b$ exists in $E$,  and the operations in direct products are defined pointwise, it is easy to deduce that PEAs with any of the properties from Definition
\ref{de:de3.1DV1} are strong type-classes.

In \cite{YY}, the following class of PEAs was introduced: An effect
algebra $E$ is \emph{weak-commutative} if, for any $a,b\in E$, $a+b$
exists iff $b+a$ exists. It is easy to see that $E$ is weak-commutative iff for all $a\in E$, $a^-=a^{\sim}$. Indeed, if $E$ is weak-commutative, from $a^-+a=1=a+a^{\sim}$ we obtain
$a+a^-$ and $a^{\sim}+a$ exist, so that $a^-\leq a^{\sim}$ and $a^{\sim}\leq a^-$. On the other hand, if $a^-=a^{\sim}$, then $a+b$ exists iff $b\leq a^{\sim}=a^-$ iff $b+a$ exists.
A weak-commutative PEA becomes an effect algebra iff $a+b=b+a$ whenever
one side of the equality exists. It was shown in \cite{YY} that
effect algebras are a proper subclass of weak-commutative
pseudo-effect algebras.

\begin{theorem}\label{th:wcpea} The class of weak-commutative PEAs
is a type-class, which is not a strong type-class.
\end{theorem}

\begin{proof} Let $c\in \Gamma(E)$, $a,b\in E[0,c]$. Then $a+b$
exists in $E[0,c]$ iff $a+b$ exists in $E$, so $b+a$ exists in $E$,
whence $b+a$ exists in $E[0,c]$. Verification of the remaining
properties of a type-class is straightforward.
Suppose that the class in question is a strong type-class. Then
for every $d\in E$, $E[0,d]$ would be weak-commutative; hence
if $a,b\leq d$ and $a+b\leq d$, then $b+a\leq d$. Putting $d=a+b$
yields $b+a\leq a+b$, and putting $d=b+a$ yields $a+b\leq b+a$.
\end{proof}

In what follows we assume that $K$ and $F$ are TD subsets of the COPEA $E$ and that $K\subseteq F$.
As in Theorem \ref{th:theorem5.9}, we decompose $E$ as
$$
E=E[0,c_1]\times E[0,c_2]\times E[0,c_3] \, \mbox{and also as }\ E=E[0,d_1]\times E[0,d_2]\times E[0,d_3]
$$
where $c_1=c_{K\cap\gamma K}$ and $d_1=c_{F\cap\gamma F}$ are of types $K$ and $F$, respectively; $c_2=c_K\wedge c_{K\cap \gamma K}'$ and $d_2=c_F\wedge c_{F\cap\gamma F}'$ are locally types $K$ and $F$, but properly non-$K$ and properly non-$F$, respectively; and $c_3=c_K'$ and $d_3=c_F'$ are purely non-$K$ and purely non-$F$, respectively.

As $K\subseteq F$, it is clear that, type-$K$ implies type-$F$; locally type-$K$ implies locally type-$F$; purely non-$F$ implies purely non-$K$; and properly non-$F$ implies properly non-$K$.

The following theorem is an analogue of \cite[Theorem 6.6]{FoPuType} proved for effect algebras; since its proof in pseudo-effect algebra setting follows the same ideas, we omit it.

\begin{theorem}\label{th:th6.2t} There exists a direct sum decomposition
$$
E=E[0,c_{11}]\times E[0,c_{21}]\times E[0,c_{22}]\times E[0,c_{31}]\times E[0,c_{32}]\times E[0,c_{33}]
$$
where $c_{11}$ is type-$K$ (hence type-$F$); $c_{21}$ is type-$F$, locally type-$K$, but properly non-$K$; $c_{22}$ is locally type-$K$ (hence, locally type-$F$), but properly non-$F$ (hence, properly non-$K$); $c_{31}$ is type-$F$ and purely non-$K$; $c_{32}$ is locally type-$F$ but properly non-$F$, and purely non-$K$; and $c_{33}$ is purely non-$F$ (hence, purely non-$K$).
Moreover, such a decomposition is unique, with $c_{ij}=c_i\wedge d_j$ for $i,j=1,2,3$, where $c_{11}=c_1$, $c_{33}=d_3$ and $c_{12}=c_{13}=c_{23}=0$.
\end{theorem}

In analogy with the classical decomposition of von Neumann algebras into types I, II, and III, we introduce the following definition (see also \cite[Definition 6.3]{FoPuType}).

\begin{definition}\label{de:I,II,III} For the TD sets $K$ and $F$ with $K\subseteq F$, the COPEA $E$ is \emph{type I} iff it is locally type-$K$; \emph{type II} iff it is locally type-$F$, but purely non-$K$; and \emph{type III} iff it is purely non-$F$.
It is type I$_F$ (respectively, type II$_F$) iff it is type I (respectively, type II) and also type-$F$. It is type I$_{{\bar F}}$ (respectively, type II$_{{\bar F}}$ iff it is of type I (respectively, type II) and also properly non-$F$.
\end{definition}

The following theorem is the I/II/III-decomposition theorem for COPEAs.

\begin{theorem}\label{th:th6.4t} $E$ decomposes as $E=E[0,c_I]\times E[0,c_{II}]\times E[0,c_{III}]$, where $c_I, c_{II}$ and $c_{III}$ are central elements of types I,II, and III, respectively; such a decomposition is unique, and $c_I=c_K$, $c_{II}=c_F\wedge c_K'$, $c_{III}=c_F'$.

Moreover, there are further decompositions $E[0,c_I]=E[0,c_{IF}]\times E[0,c_{I{\bar F}}]$ and $E[0,c_{II}]=E[0,c_{IIF}]\times E[0,c_{II{\bar F}}]$, where $c_{IF}, c_{I{\bar F}}, c_{IIF}, c_{II{\bar F}}$ are central elements of types $I_F, I_{{\bar F}}, II_{F}, II_{{\bar F}}$, respectively; these decompositions are also unique.
\end{theorem}

These decompositions are obtained if in Theorem \ref{th:th6.2t} we put $c_I:=c_{11}+ c_{21}+ c_{22}$; $c_{II}:=c_{31}+ c_{32}$ and $c_{III}=c_{33}$; $c_{IF}:= c_{11}+c_{21}$, $c_{I{\bar F}}:=c_{22}$, $c_{IIF}:=c_{31}$,
$c_{II{\bar F}}:=c_{32}$. Notice that, beyond the traditional I/II/III decomposition,  the type I$_F$ summand decomposes as $E[0,c_{IF}]=E[0,c_{11}]\times E[0,c_{21}]$, where
$c_{11}$ is type-$K$ (hence type-$F$) and $c_{21}$ is is type -$F$ and locally type-$K$, but properly non-$K$.

\begin{example}\label{ex:ex1}
Taking $K:=[A]$, the set of all polyatoms, and $F:=H$, the set of all monads of $E$, in Theorem \ref{th:th6.4t}, we have $[A]\subseteq H$, and $E$ decomposes as $E=E[0,r_1]\times E[0,r_2]\times E[0,r_3]$
where every nonzero direct summand of $E[0,r_1]$ contains an atom; $E[0,r_2]$ is atom free, but every nonzero direct summand of $E[0,r_2]$ contains a nonzero monad; and $E[0,r_3]$ contains no nonzero monad. This decomposition is unique. Indeed, $r_1=c_{[A]}$ is locally type-$[A]$; $r_2=c_H\wedge c_{[A]}'$ is locally type-$H$ and purely non-$[A]$, and  $r_3=c_H'$ is purely non-$H$ (see Theorem \ref{th:th6.4t} and Corollary \ref{co:co5.4t}).
\end{example}

\begin{example}\label{ex:ex2}
Take $K=: EA$, the subset of all elements $e\in E$ such that $E[0,e]$ is commutative PEA (i.e., an effect algebra),
and $F=:W$, the set of all elements $d\in E$ such that $E[0,d]$ is weak-commutative. Then $EA\subseteq W$, and we obtain the decomposition $E=E[0,v_1]\times E[0,v_2]\times E[0,v_3]$.
The summand $E[0,v_1]$ is locally commutative in the sense that
 $v_1=\gamma e=c_{EA}$; the summand $E[0,v_2]$ is locally weak-commutative, but purely non-commutative, that is, $v_2=c_W\wedge c_{EA}'$; and $E[0,v_3]$ is purely non-weak-commutative, that is, $v_3=c_W'$.
We recall that then every direct sub-summand of $E[0,v_1]$ contains an element $e\in EA$; every direct sub-summand of $E[0,v_2]$ contains an element $d\in W$, but
$E[0,v_2]\cap EA=\{ 0\}$; and $E[0,v_3]$ contains no element of $W$.

The summands $E[0,v_1]$  and $E[0,v_2]$ decompose further into weak-commutative and properly non-weak-commutative parts; and the weak-commutative part of $E[0,v_1]$ admits a further decomposition into a commutative and a locally commutative, but properly non-commutative parts.
\end{example}

Let $R2$ denote the STD of elements $e\in E$ such that $E[0,e]$ satisfies (RDP$_2$) and $L$ denote the set of elements $e\in E$ such that $E[0,e]$ is a lattice.

\begin{example}\label{ex:ex3}
There exists a decomposition $E=E[0,c_{11}]\times E[0,c_{21}]\times E[0,c_{22}]\times  E[0,c_{31}]\times E[0,c_{32}]\times E[0,c_{33}]$ where $E[0,c_{11}]$ satisfies (RDP$_2$), hence is a lattice; $E[0,c_{21}]$ is a lattice, every direct sub-summand contains an element from $R2$, but no direct sub-summand satisfies (RDP$_2$); $E[0,c_{22}]$ contains no lattice ordered direct sub-summand (hence no sub-summand satisfying (RDP$_2$)), but every direct sub-summand contains an element from $R2$ (hence from $L$); $E[0,c_{31}]$ is a lattice and contains no element from $R2$;
$E[0,c_{32}]$ contains no lattice ordered direct sub-summand, and no element from $R2$,
 but every direct sub-summand contains an element from $L$; and $E[0,c_{33}]$ contains no element from $L$ (hence no element from $R2$).
Moreover, such a decomposition is unique.

Indeed, such a decomposition is obtained from decompositions corresponding to STD sets $R2$ and $L$ as in Theorem \ref{th:th6.4t}, taking into account that $R2\subseteq L$ by proposition \ref{pr:pr3.3DV1} (ii).

Notice that by \cite{DV2}, a pseudo-effect algebra satisfying (RDP$_2$) is a pseudo-MV algebra (a non-commutative analogue of an MV-algebra, see \cite{R, GI}).

\end{example}

\end{document}